\documentclass[preprint]{imsart}

\usepackage{amsmath,amsthm,amsfonts,amscd,amssymb,amstext,bbm}

\usepackage{xparse,mathrsfs,mathtools,todonotes,listings}
\usepackage[colorlinks,citecolor=blue]{hyperref}

\startlocaldefs

\newtheorem{theorem}{Theorem}
\newtheorem{lemma}[theorem]{Lemma}
\newtheorem{corollary}[theorem]{Corollary}
\newtheorem{remark}[theorem]{Remark}
\newtheorem{definition}[theorem]{Definition}

\numberwithin{equation}{section}
\numberwithin{theorem}{section}

\newcommand{\dif}{\mathrm{d}}

\newcommand{\D}{\mathbb{D}}

\newcommand{\C}{\mathbb{C}}
\newcommand{\Z}{\mathbb{Z}}

\newcommand{\E}{\mathbb E}
\newcommand{\N}{\mathbb N}

\newcommand{\Exp}{\mathbb{E}}

\newcommand{\filt}{\mathscr{F}}
\newcommand{\Gfilt}{\mathscr{G}}
\DeclareDocumentCommand \one { o }
{%
  \IfNoValueTF {#1}
  {\mathbf{1}  }
  {\mathbf{1}\left\{ {#1} \right\} }%
}

\DeclareDocumentCommand{\Prto} {o O{\infty} } {
  \IfNoValueTF {#1}
  {\overset{p}{\longrightarrow}}
  { \xrightarrow[ #1 \to #2]{p }}
}
\DeclareDocumentCommand{\Asto} {o} {
  \IfNoValueTF {#1}
  {\overset{\operatorname{a.s.}}{\longrightarrow}}
  {
    \xrightarrow[ #1 \to \infty]{\operatorname{a.s.} }
  }
}
\DeclareDocumentCommand{\Mgfto} {o} {
  \IfNoValueTF {#1}
  {\overset{\operatorname{mgf}}{\longrightarrow}}
  { \xrightarrow[ #1 \to \infty]{\operatorname{mgf} }}
}

\DeclareDocumentCommand{\Wkto} {o} {
  \IfNoValueTF {#1}
  {\overset{d}{\longrightarrow}}
  { \xrightarrow[ #1 \to \infty]{d }}
}

\newcommand{\mtp}{\mathcal{M}_{\theta}(|p^2|) }

\endlocaldefs

\begin{document}
\begin{frontmatter}

\title{The Fourier coefficients of the holomorphic multiplicative chaos in the limit of large frequency}
\runtitle{HMC in the $L^{1}$-phase}
\runauthor{Najnudel, Paquette, Simm and Vu}
\begin{aug}
  \author[A]{
    \fnms{Joseph}
    \snm{Najnudel}
    \ead[label=e1,mark]{joseph.najnudel@bristol.ac.uk}}
  \author[B]{
    \fnms{Elliot}
    \snm{Paquette}
    \ead[label=e2,mark]{elliot.paquette@mcgill.ca}}
  \author[C]{
    \fnms{Nick}
    \snm{Simm}
    \ead[label=e3,mark]{n.j.simm@sussex.ac.uk}}
  \author[D]{
      \fnms{Truong}
      \snm{Vu}
      \ead[label=e4,mark]{tvu25@uic.edu}}

 \address[A]{School of Mathematics; University of Bristol, BS8 1UG, United Kingdom; \printead{e1}}
 \address[B]{Department of Mathematics and Statistics; McGill University, Montreal, Quebec, Canada; \printead{e2}}
 \address[C]{Department of Mathematics; University of Sussex, BN1 9RH, United Kingdom; \printead{e3}}
 \address[D]{Department of Mathematics, Statistics, and Computer Science, University of Illinois at Chicago, Chicago, Illinois, USA; \printead{e4}}
\end{aug}

\begin{abstract}
  The holomorphic multiplicative chaos (HMC) is a holomorphic analogue of the Gaussian multiplicative chaos.  It arises naturally as the limit in large matrix size of the characteristic polynomial of Haar unitary, and more generally circular-$\beta$-ensemble, random matrices.
  We consider the Fourier coefficients of the holomorphic multiplicative chaos in the $L^1$-phase, and we show that appropriately normalized, this converges in distribution to a complex normal random variable, scaled by the total mass of the Gaussian multiplicative chaos measure on the unit circle.  We further generalize this to a process convergence, showing the joint convergence of consecutive Fourier coefficients.  As a corollary, we derive convergence in law of the secular coefficients of sublinear index of the circular-$\beta$-ensemble for all $\beta > 2$.
\\\\\\\\
\end{abstract}



\end{frontmatter}

\section{Introduction}

In this paper we consider the Holomorphic multiplicative chaos. This is a distribution on the unit circle, arising in the theory of random matrices and analytic number theory. See previous work of the first three authors \cite{najnudel2023secular} where it was introduced and studied. To define it, let $\{\mathcal{N}_{k}\}_{k=1}^{\infty}$ be i.i.d.\ \emph{standard complex normal} random variables, i.e.\ having moments
\begin{equation}
  \Exp(\mathcal{N}_k) = 0,
  \quad
  \Exp (\mathcal{N}_k^{2}) = 0,
  \quad
  \text{and}
  \quad
  \Exp(|\mathcal{N}_k|^2) = 1. \label{iidnk}
\end{equation}
Let $G^{\C}$ be the Gaussian analytic function on the unit disc $\D$, 
\begin{equation}
  G^{\C}(z) = \sum_{k=1}^\infty \frac{z^k}{\sqrt{k}}\mathcal{N}_k.
  \label{eq:GAF}
\end{equation}

For trigonometric polynomials $\phi$ if we let $z \mapsto \phi(z)$ for $z \in \D$ be the continuous harmonic extension to $\D,$ we then define the random distribution
\begin{equation}
  ( \mathrm{HMC}_\theta, \phi)
  \coloneqq
  \lim_{r \to 1} 
  \frac{1}{2\pi}
  \int_0^{2\pi} 
  e^{\sqrt{\theta} G^{\C}(re^{i\vartheta})}
%  \Phi_\infty^*(re^{i\vartheta}) 
  \overline{\phi(r\vartheta)} \dif\vartheta,
  \label{eq:hmc}
\end{equation}
which is the \emph{holomorphic multiplicative chaos}
in terms of the inverse temperature parameter
\begin{equation}
\theta \coloneqq \frac{2}{\beta}.
\end{equation}

We define for $n \in \N_0$ the Fourier coefficient of the HMC
\begin{equation}\label{eq:cn}
  c_n \coloneqq ( \mathrm{HMC}_\theta, \vartheta \mapsto  e^{i n\vartheta}).
\end{equation}
Equivalently, the coefficients $c_{n}$ can be extracted from a generating function by the formula
\begin{equation}
c_{n} = [z^{n}]\,e^{\sqrt{\theta}G^{\mathbb{C}}(z)} = [z^{n}]\,\mathrm{exp}\left(\sqrt{\theta}\sum_{k=1}^{\infty}\frac{z^{k}}{\sqrt{k}}\,\mathcal{N}_{k}\right), \label{cngen}
\end{equation}
where the notation $[z^{n}]\,h(z)$ denotes the coefficient of $z^{n}$ in the power series expansion of $h(z)$ around the point $z=0$. We could as well define this for $n \in \Z,$ but for negative integers, this would be $0$. 

The coefficients \eqref{cngen} were studied in detail in \cite{najnudel2023secular}, where a convergence in distribution result is shown for any $0 < \theta < \frac{1}{2}$. The recent work \cite{gorodetsky2024martingale} established analogues of this convergence for a class of random multiplicative functions with the same restriction $0 < \theta < \frac{1}{2}$. Although we do not consider it here, the case $\theta=1$ has attracted recent attention due to its connections to analytic number theory, in particular the phenomenon of `better than square root cancellation'.  In the critical $\theta=1$ case, the na\"ive bound $\mathbb{E}(|c_{n}|) \leq \sqrt{\mathbb{E}(|c_{n}|^{2})}$ overestimates the true order of magnitude in $n$. Sharp upper and lower bounds on this quantity are derived in \cite{soundzaman}; see also \cite{najnudel2023secular} for another proof of the upper bound.  

Moreover, for various randomized analytic number theory problems, such as random multiplicative functions, better-than-square-root cancellation has been proven \cite{H20,gerspach2022almost,soundzaman,xu2024better, gorodetsky2024short}. See also \cite{harper2024moments} for a recent review and see \cite{gu2024universality} for related work on non-Gaussian holomorphic multiplicative chaoses. 

For $0 < \theta < 1$, the Gaussian multiplicative chaos can be defined via \eqref{eq:GAF} as 
  \begin{equation}
    \mathrm{GMC}_\theta(\dif\vartheta) \coloneqq 
  %\lim_{r \to \infty} (1-r^2)^{\theta}| \Phi_\infty^*(r e^{i \vartheta})|^2 d\vartheta
  \lim_{r \to 1} (1-r^2)^{\theta}| e^{\sqrt{\theta}G^{\C}(r e^{i\vartheta})}|^2 \dif\vartheta
  =\lim_{r \to 1} (1-r^2)^{\theta}e^{\sqrt{\theta}G(r e^{i\vartheta})} \dif\vartheta,
  \label{eq:gmc}
\end{equation}
where $G(z) = 2\mathrm{Re}(G^{\mathbb{C}}(z))$, with the sense of the limit being an in-probability, weak-$*$ convergence. The existence of this limit as a random measure is shown\footnote{The Gaussian field used in \cite[Eqn 6.1]{JunnilaSaksman} contains a constant term in its Fourier series definition in comparison to our field $G$. This turns out not to affect the convergence results of \cite{JunnilaSaksman}, see \cite[Appendix B]{najnudel2023secular}.} in \cite{JunnilaSaksman} where the convergence is shown to hold in $L^{q}$ for any $0<q<1$ (c.f.\ \cite[Proposition 3.1]{ChhaibiNajnudel}). Moreover it is shown in \cite{Remy} that the total mass of this particular random measure has law characterized by the natural analytic continuation of the Morris integral, i.e.\ for any $p >0$ with $p \theta < 1,$
\begin{equation}\label{eq:mass} 
  \Exp\bigl(\mathcal{M}_\theta^p\bigr)
  =\frac{\Gamma(1-p\theta)}{\Gamma(1-\theta)^{p}},
  \quad
  \text{where}
  \quad
  \mathcal{M}_\theta
  \coloneqq
  \frac{1}{2\pi}\int_0^{2\pi}\mathrm{GMC}_{\theta}(\dif\vartheta).
\end{equation}

The main goal of this paper is to prove:
\begin{theorem}[$L^{1}$-phase]
  For any $0 < \theta < 1$, we have the convergence in distribution
  \begin{equation}
    \frac{c_{n}}{\sqrt{\mathbb{E}(|c_{n}|^{2})}} \Wkto[n] \sqrt{\mathcal{M}_{\theta}}\mathcal{Z} \label{l1phaselim}
  \end{equation}
  where $\mathcal{Z}$ and $\mathcal{M}_{\theta}$ are independent, $\mathcal{Z}$ is standard complex normal, and $\mathcal{M}_{\theta}$ has law \eqref{eq:mass}.
  \label{th:l1phase}
\end{theorem}
\noindent This improves over \cite{najnudel2023secular} which establishes the same theorem in the range $\theta$ from $0 < \theta < \tfrac12$, and it now covers the full range of subcritical HMC.  In the case $\theta=1$, we expect that after rescaling the prelimit by $(\log n)^{1/4}$ there is again convergence in distribution to a random variable connected to the critical Gaussian multiplicative chaos; see also \cite{aggarwal2022conjectural} for simulations related to this convergence.  

We note there have been recent developments on the Fourier coefficients of the real GMC, which display a formula similar to \eqref{eq:mass}; in \cite{garban2023harmonic}, they show the $n$-th Fourier coefficient of $\mathrm{GMC}_\theta(    \dif \vartheta)$ converges in law (after renormalization) to a complex normal multiplied by $\sqrt{\mathcal{M}_{2\theta}}$ for $2\theta < 1$.

\subsection{Process convergence}

The proof we give extends to finite-dimensional marginal convergence of the $\{c_{n+k} : k \in \Z\}$ in the limit of $n\to\infty$.  To formulate this convergence, we can equivalently consider a fixed complex polynomial $p$ and define the statistic 
\[  
  X_n[p] \coloneqq  
  (\mathrm{HMC}_\theta, \vartheta \mapsto e^{i n \vartheta}p(e^{i \vartheta})).
\]
A small modification of the proof of Theorem \ref{th:l1phase} gives rise to the following generalization.  
\begin{theorem}[$L^{1}$-phase-multipoint]
  For any $0 < \theta < 1$, we have the convergence in distribution
  \begin{equation}
    \frac{
      \left( 
        X_{n}[p]
      \right)
      }{\sqrt{\mathbb{E}(|c_{n}|^{2})}} \Wkto[n] 
    \sqrt{\mtp}
    \mathcal{Z} \label{l1phaselimp}
  \end{equation}
  where $\mathcal{Z}$ is independent of $\mathrm{GMC}_{\theta}$, $\mathcal{Z}$ is standard complex normal, and where
  \[
    \mtp
    \coloneqq
    \frac{1}{2\pi}\int_0^{2\pi}|p(e^{i \vartheta})|^2 \mathrm{GMC}_{\theta}(\dif \vartheta).
  \]
  \label{th:l1phasep}
\end{theorem}
\noindent We shall prove Theorem \ref{th:l1phasep}, from which Theorem \ref{th:l1phase} follows as a direct corollary. Another corollary is the joint convergence of the shifted coefficients $\{c_{n+k}\}_{k=1}^{\ell}$.

\begin{corollary}
\label{cor:shiftedcoeffs}
  Fix an $\ell \in \N$.
  Define the positive-definite Toeplitz matrix 
  \[
    \mathcal{H}_{k_1,k_2} = \frac{1}{2\pi} 
    \int_{0}^{2\pi}
    e^{i \vartheta (k_1-k_2)}
    \mathrm{GMC}_{\theta}(\dif \vartheta),
    \quad 0 \leq k_1,k_2 \leq \ell.
  \]
  For any $0  <\theta < 1$, we have the convergence in distribution
  \[
   \frac{(c_{n+k}: 0 \leq k \leq \ell )}
   { \sqrt{ \Exp | c_n|^2} }
   \Wkto[n]
   \sqrt{\mathcal{H}}( \mathcal{Z}_k : 0 \leq k \leq \ell),
  \]
  where $(\mathcal{Z}_k : 0 \leq k \leq \ell)$ are i.i.d.\ standard complex normal random variables, independent of $\mathrm{GMC}_{\theta}$.
\end{corollary}

\subsection{The circular $\beta$-ensemble}

One way in which the HMC arises is through the limit of the characteristic polynomial of a random matrix.  The \emph{circular $\beta$-ensemble} is the joint density  $N$ points $\vartheta_{j} \in [0,2\pi)$ for $1 \leq j \leq N$ given by
\begin{equation}
  \text{C$\beta$E}_N(\vartheta_1,\ldots,\vartheta_N) \propto \prod_{1 \leq k < j \leq N}|e^{i\vartheta_k}-e^{i\vartheta_j}|^{\beta} \label{betaens}
\end{equation}
When $\beta=2$, i.e. $\theta = \tfrac{2}{\beta}=1,$ this distribution arises as the law of the eigenvalues of a Haar distributed unitary matrix and is better known as the CUE (Circular Unitary Ensemble). When $\beta \neq 2$ it also arises from the eigenvalue distribution of certain sparse random matrix models, see \cite{KillipNenciu}.  

The characteristic polynomial of these points is given by
\begin{equation}
  \chi_{N}(z) = \prod_{j=1}^{N}(1-ze^{-i\vartheta_j}), \label{charpoly}
\end{equation}
which would have the representation $\det(1-zU^*_{N})$ if $\{e^{i\vartheta_j} : 1 \leq j \leq N\}$ were the eigenvalues of the matrix $U_N.$

The \textit{secular coefficients}, defined most simply as the coefficients $c^{(N)}_{n}$ in the expansion of \eqref{charpoly} in its Fourier basis,
\begin{equation}
\chi_{N}(z) = \sum_{n=0}^{N}c^{(N)}_{n}z^{n} \label{chisec}
\end{equation}
can be connected to the Fourier coefficients $\{c_n\}$.  Moreover, from \cite{ChhaibiNajnudel}, there is a probability space realizing all $\{c_n^{(N)}\}$ so that 
\[
  c_n^{(N)} \Asto[N] c_n  \quad \text{for all } n \geq 0.
\]

From \cite[Lemma 7.2]{najnudel2023secular}, for any $\theta < 1$ there is a constant $C_\theta$ 
\[
  \Exp |c_n -   c_n^{(N)}|^2 \leq C_\theta \left(\Exp|c_n|^2\right)\frac{n}{N}.
\]
Thus the following corollary follows immediately.
\begin{corollary}
  For any $\theta < 1$ and any $N_n \to \infty$ with $n/N_n \to 0$ as $n \to \infty$, 
  \[
  \frac{(c^{(N_n)}_{n+k}: 0 \leq k \leq \ell )}
  { \sqrt{ \Exp | c_n|^2} }
  \Wkto[n]
  \sqrt{\mathcal{H}}( \mathcal{Z}_k : 0 \leq k \leq \ell).
  \]
\end{corollary}

\section*{Acknowledgements}

The first three authors were supported by the Swedish Research Council under grant no. 2021-06594 while in residence at Institut Mittag-Leffler in Djursholm, Sweden during the fall semester of 2024.  EP would like to acknowledge the support from a Canada NSERC Discovery grant. NS would like to acknowledge support from the Royal Society, grant URF\textbackslash R\textbackslash231028.

\section{Preliminaries from \cite{najnudel2023secular} and the Ewens sampling formula}
\label{sec:ewens}
By formula \eqref{cngen}, the secular coefficients $c_{n}$ can be expressed as
\begin{equation}
  c_{n} = \sum_{\vec{m} \in S_{n}}\,\prod_{k=1}^{n}\frac{\mathcal{N}_{k}^{m_{k}}}{m_{k}!}\,\left(\frac{\theta}{k}\right)^{m_{k}/2}. \label{cnsum}
\end{equation}
The summation is over the set $S_{n}$ of all compositions of $n$, that is $\vec{m} = (m_1,\ldots,m_n)$ is such that each $m_{k}$ is a non-negative integer satisfying
\begin{equation}
  \sum_{k=1}^{n}km_{k} = n.
\end{equation}
From \eqref{cnsum}, we have that $c_n$ is measurable with respect to $\Gfilt_n \coloneqq \sigma\left\{ \mathcal{N}_1, \dots, \mathcal{N}_n \right\}.$  So, going forward, we will make use of the filtration
\( \Gfilt = ( \Gfilt_n : n \in \N)\). We also make use of the following truncated coefficients $c_{n,q}$ defined as in \eqref{cnsum} but setting $\mathcal{N}_{k}=0$ for $k>q$:
\begin{equation}
  c_{n,q} \coloneqq \sum_{\substack{(m_{k}) : 1 \leq k \leq q\\ \sum_{k=1}^{q}km_{k}=n}}\prod_{k=1}^{q}\frac{\theta^{m_{k}/2}}{m_{k}!k^{m_k/2}}\,\mathcal{N}_{k}^{m_k}. \label{cnq}
\end{equation}
Taking the $L^{2}$-norm of \eqref{cnsum} gives
\begin{equation}
\begin{split}
  \mathbb{E}(|c_{n}|^{2}) &= \sum_{\substack{\vec{l} \in S_{n},\vec{m}\in S_{n}}}\prod_{k=1}^{n}\frac{\mathbb{E}(\mathcal{N}_{k}^{m_k}\overline{\mathcal{N}_{k}}^{l_k})}{m_{k}!l_{k}!}\left(\frac{\theta}{k}\right)^{(m_{k}+l_{k})/2} \label{l2first}\\
  &=\sum_{\substack{\vec{m}\in S_{n}}}\prod_{k=1}^{n}\frac{1}{m_{k}!}\left(\frac{\theta}{k}\right)^{m_{k}}
  \end{split}
\end{equation}
where we used independence of the family $\{\mathcal{N}_{k}\}_{k=1}^{\infty}$ and the Gaussian formula
\begin{equation}
  \mathbb{E}(\mathcal{N}_{k}^{m_k}\overline{\mathcal{N}_{k}}^{l_k}) = \mathbbm{1}_{m_k=l_k}(m_k)! \label{gausscomp2}.
\end{equation} 
Given a permutation $\sigma$ on $n$ symbols, we can characterize it using its cycle structure $(m_1,\ldots,m_n)$ where $m_{j}$ denotes the number of cycles in $\sigma$ having length $j$, and $\sum_{j=1}^{n}jm_{j} = n$.
\begin{definition}
  The \textit{Ewens sampling formula} is a probability distribution on cycle counts $\vec{M}^{(n)} = (M_1,\ldots,M_n)$ given by
  \begin{equation}
    \mathbb{P}(\vec{M}^{(n)} = \vec{m}) = 
    \frac{
    \mathbbm{1}_{\sum_{k=1}^{n}km_{k}=n}\,
    }{\binom{n+\theta-1}{\theta-1}}
%    \frac{n!}{\theta^{(n)}}
    \prod_{k=1}^{n}\left(\frac{\theta}{k}\right)^{m_k}\frac{1}{m_{k}!}. \label{ewens}
  \end{equation}
\end{definition}
By the normalization of \eqref{ewens} and \eqref{l2first}, this gives the $L^{2}$ norm explicitly
\begin{equation}
\mathbb{E}(|c_{n}|^{2}) = \binom{n+\theta-1}{\theta-1}.
\end{equation}
A straightforward generalization of \eqref{l2first} gives, for $0 \leq q \leq n$,
\begin{equation}
\mathbb{E}(c_{n-q,q_{1}}\overline{c_{n-q,q_{2}}}) = \binom{n-q+\theta-1}{\theta-1}\mathbb{P}(L^{(n-q)}\leq q_1\wedge q_2), \label{covcnq}
\end{equation}
where $L^{(n)}$ denotes the longest cycle in a Ewens distributed permutation of length $n$. The properties of the Ewens measure are surveyed in detail in the text \cite{ABT03}. A key result used therein is that \eqref{ewens} is the joint law of $n$ independent random variables $(Z_1,\ldots,Z_{n})$ where each $Z_{k}$ is Poisson distributed with parameter $\theta/k$ subject to the condition that $T_{0n}=n$ where $T_{0n} = \sum_{j=1}^{n}jZ_{j}$. In the following we quote the limiting distributions of the random variables $T_{0n}$ and $L^{(n)}$.
\begin{lemma}(\cite[Section 4]{ABT03})
  \label{lem:t0n}
  Suppose that $r = r_{n} \in \mathbb{N}$ satisfies $r/n \to y \in (0,\infty)$ as $n \to \infty$. Then
  \begin{equation}
    \lim_{n \to \infty}n\mathbb{P}(T_{0n}=r) = p_{\theta}(y) \label{convt0n}
  \end{equation}
where $p_{\theta}(y)$ is an explicit probability density function with rapid tail decay $\sup_{y \geq n}p_{\theta}(y) \leq \theta^{n}/n!$.
  \end{lemma}
The function $p_{\theta}(y)$ also appears in the following result \cite[Lemma 4.23]{ABT03}, attributed to Kingman (1977).
\begin{lemma}
  \label{lem:kingman}
  As $n \to \infty$, we have the convergence in distribution $n^{-1}L^{(n)} \overset{d}{\longrightarrow} L^{(\infty)}$ where $L^{(\infty)}$ is a random variable with distribution function $F_{\theta}$ given by
  \begin{equation}
    F_{\theta}(x) = e^{\gamma_{\mathrm{E}}\theta}x^{\theta-1}\Gamma(\theta)p_{\theta}(1/x)
  \quad \text{for all } x \in (0,1]. \label{kingman}
  \end{equation}
\end{lemma}
We will need the following constant defined in terms of $L^{(\infty)}$.
\begin{lemma}(\cite[Lemma 4.5]{najnudel2023secular})
\label{lem:cdef}
  Consider the quantity
  \begin{equation}
    C_{\delta} \coloneqq \theta\int_{\delta}^{1}(1-x)^{\theta-1}\mathbb{P}\left(L^{(\infty)} \leq \frac{x}{1-x}\right)\,\frac{\dif x}{x} \label{cdeldef}
  \end{equation}
  where $L^{(\infty)}$ is the limiting random variable from Lemma \ref{lem:kingman}. Then $C_{\delta}$ can be explicitly computed as
  \begin{equation}
    C_{\delta} = 1-\Gamma(\theta)e^{\gamma_{\mathrm{E}}\theta}\,\delta^{\theta-1}p_{\theta}(1/\delta), \label{cdelident}
  \end{equation}
 where $\gamma_{\mathrm{E}}$ is the Euler-Mascheroni constant. We have
  \begin{equation}
    C_{\delta} = 1+O(\delta), \qquad \delta \to 0. \label{cto1}
  \end{equation}
\end{lemma}
These facts are used in \cite[Lemma 4.1]{najnudel2023secular} to prove the following.
\begin{lemma}
  \label{le:martingalerep}
  Let $0 < \delta < 1$ and assume $\theta \in [0,1]$. Define
  \begin{equation}
  \tilde{c}_{n}^{(\delta)} \coloneqq \sum_{q=\lfloor \delta n \rfloor}^{n}\mathcal{N}_{q}\sqrt{\frac{\theta}{q}}\,c_{n-q,q-1} \label{martingalesum}
\end{equation}
Then
  \begin{equation}
    \lim_{\delta \to 0}\lim_{n \to \infty}\frac{\mathbb{E}(|c_{n}-\tilde{c}_{n}^{(\delta)}|^{2})}{\mathbb{E}(|c_{n}|^{2})} = 0.
  \end{equation}
\end{lemma}

\section{Martingale approximation and convergence}
\label{sec:mgle1}
In this section we begin by finding an alternative martingale structure in the sum \eqref{cnsum}. This opens up the possibility of applying a central limit theorem for martingales, and we explain why this is relevant for our proof of Theorem \ref{th:l1phase}. In contrast to \cite{najnudel2023secular}, we show that the bracket process converges in the whole $L^{1}$ phase $0<\theta<1$.
\subsection{Martingale approximation}
Note that $\tilde{c}_{n}^{(\delta)}$ in \eqref{martingalesum} is a martingale, for any fixed $n,\delta$, when viewed as sum over $q$ of martingale increments and with respect to the filtration \( \Gfilt = ( \sigma\{\mathcal{N}_{1},\ldots,\mathcal{N}_{q}\} : q \in \N)\). This follows from the fact that $c_{n-q,q-1}$ only depends on the first $q-1$ Gaussians. In this section, we make a further simplification of $\tilde{c}_{n}^{(\delta)}$. Let $\epsilon > 0$ be another fixed parameter and let 
\[
n_k \coloneqq  \left( \lfloor \delta n \rfloor  + k \lfloor \epsilon n \rfloor \right) \wedge n
\]
for $k \in \N_0$.  We let $K$ be the smallest $k$ such that $n_k = n$; note that for fixed $\delta,\epsilon$, $K$ is almost fixed in that it is nearly the smallest $k$ such that $\delta+k \epsilon \geq 1$, up to $\pm 1$. Then we define 
\begin{equation}
  \tilde{c}_{n}^{(\delta,\epsilon)} 
  \coloneqq   \sum_{k=0}^{K-1}
  \sum_{q = n_{k}+1}^{n_{k+1}}
  \mathcal{N}_{q}
  \sqrt{\frac{\theta}{n_k}}\,c_{n-q,n_{k}}. 
  \label{martingalesum2024}
\end{equation}
Note that $\tilde{c}_{n}^{(\delta,\epsilon)}$ continues to satisfy the martingale property in $q$. Next we show further:
\begin{lemma}[Martingale approximation]
  \label{le:martingalerep2024}
  Let $0 < \delta,\epsilon < 1$ and assume $\theta \in [0,1]$. We have
  \begin{equation}
    \lim_{\delta \to 0}
    \lim_{\epsilon \to 0}
    \lim_{n \to \infty}\frac{\mathbb{E}(|c_{n}-\tilde{c}_{n}^{(\delta,\epsilon)}|^{2})}{\mathbb{E}(|c_{n}|^{2})} = 0.
  \end{equation}
\end{lemma}
\begin{proof}
By Lemma \ref{le:martingalerep}, it is sufficient to show the same for $\mathbb{E}(|\tilde{c}^{(\delta)}_{n}-\tilde{c}_{n}^{(\delta,\epsilon)}|^{2})$. We use the martingale property to write
\begin{equation}
\mathbb{E}(|\tilde{c}_{n}^{(\delta)}-\tilde{c}_{n}^{(\delta,\epsilon)}|^{2}) = \theta\sum_{k=0}^{K-1}\sum_{q=n_{k}+1}^{n_{k+1}}\mathbb{E}\left(\bigg{|}c_{n-q,q-1}/\sqrt{q}-c_{n-q,n_{k}}/\sqrt{n_{k}}\bigg{|}^{2}\right).
\end{equation}
We expand the square and compute the three terms 
\begin{align}
\mathcal{S}_{1} &:= \theta\sum_{k=0}^{K-1}\sum_{q=n_{k}+1}^{n_{k+1}}\frac{1}{q}\,\mathbb{E}(|c_{n-q,q-1}|^{2}),\\
\mathcal{S}_{2} &:= \theta\sum_{k=0}^{K-1}\sum_{q=n_{k}+1}^{n_{k+1}}\frac{1}{\sqrt{q}\sqrt{n_{k}}}\,\mathbb{E}(c_{n-q,q-1}\overline{c_{n-q,n_{k}}}),
\end{align}
and
\begin{equation}
\mathcal{S}_{3} := \theta\sum_{k=0}^{K-1}\sum_{q=n_{k}+1}^{n_{k+1}}\frac{1}{n_{k}}\,\mathbb{E}(|c_{n-q,n_{k}}|^{2}).
\end{equation}
We will prove the Lemma by showing that all three terms, when rescaled, converge to the same constant,
\begin{equation}
\frac{\mathcal{S}_{j}}{\mathbb{E}(|c_{n}|^{2})} \to {C(\delta)}, \qquad j=1,2,3, \label{goalsj}
\end{equation}
in the limit $n \to \infty$ followed by $\epsilon \to 0$, where $C(\delta)$ is the explicit constant in \eqref{cdeldef} and where we recall $\mathbb{E}(|c_{n}|^{2}) \sim n^{\theta-1}/\Gamma(\theta)$ is the asymptotics of the normalizing factor. Note that we already verified \eqref{goalsj} for $j=1$ in \cite[Lemma 4.5]{najnudel2023secular}. For $j=2$, formula \eqref{covcnq} gives
\begin{equation}
\mathbb{E}(c_{n-q,q-1}\overline{c_{n-q,n_{k}}}) = \binom{n-q+\theta-1}{\theta-1}\mathbb{P}(L^{(n-q)}\leq n_{k}),
\end{equation}
where we used that $n_{k} \leq q-1$. So
\begin{equation}
\mathcal{S}_{2} = \theta\sum_{k=0}^{K-1}\sum_{q=n_{k}+1}^{n_{k+1}}\frac{1}{\sqrt{q}\sqrt{n_{k}}}\binom{n-q+\theta-1}{\theta-1}\mathbb{P}(L^{(n-q)}\leq n_{k}). \label{st2int}
\end{equation}
We use the asymptotics of the binomial coefficient for large $n-q$ together with Lemma \ref{lem:kingman} and approximate the sum over $q$ as a Riemann integral. We justify this with dominated convergence arguments as follows. The sum over $q$ in \eqref{st2int} is equal to
\begin{equation}
\frac{\sqrt{n}}{\sqrt{n_{k}}}n^{\theta-1}\int_{n_{k}/n+1/n}^{n_{k+1}/n}f_{n}(\lfloor nx \rfloor) \dif x
\end{equation}
where 
\begin{equation}
f_{n}(\lfloor nx \rfloor) = \frac{\sqrt{n}}{\sqrt{\lfloor nx \rfloor}}n^{1-\theta}\binom{n-\lfloor nx \rfloor+\theta-1}{\theta-1}\mathbb{P}(L^{(n-\lfloor nx\rfloor)}\leq n_{k}).
\end{equation}
We make use of a classical bound on the ratio of two Gamma functions, see \cite[5.6.4]{NIST:DLMF}
\begin{equation*}
\frac{\Gamma(x+\theta)}{\Gamma(x+1)} < x^{\theta-1}, \qquad x>0, \quad \theta \in (0,1)
\end{equation*}
which gives the uniform bound
\begin{equation}
n^{1-\theta}\binom{n-\lfloor nx \rfloor+\theta-1}{\theta-1} \leq (1-x)^{\theta-1}, \quad n \in \mathbb{N}.
\end{equation}
Then $f_{n}(\lfloor nx \rfloor)$ is uniformly bounded by $2(1-x)^{\theta-1}$, which is integrable for all $\theta>0$. Furthermore, for each fixed $0<x<1$, we have $f_{n}(x) \to f(x)$ as $n \to \infty$, where
\begin{equation}
f(x) = \frac{1}{\Gamma(\theta)}\frac{1}{\sqrt{x}}(1-x)^{\theta-1}\mathbb{P}\left(L^{(\infty)} \leq \frac{\delta+k\epsilon}{1-x}\right).
\end{equation}
Then using the definition of $n_{k}$, the dominated convergence theorem gives
\begin{equation}
\mathcal{S}_{2} \sim \frac{n^{\theta-1}}{\Gamma(\theta)}\,\theta\sum_{k=0}^{K-1}\int_{\delta+k\epsilon}^{(\delta+(k+1)\epsilon)\wedge 1}\frac{1}{\sqrt{x}\sqrt{\delta+k\epsilon}}(1-x)^{\theta-1}\mathbb{P}\left(L^{(\infty)} \leq \frac{\delta+k\epsilon}{1-x}\right)\,\dif x. \label{ksumapprox}
\end{equation}
Now similarly, we view the sum over $k$ as approximating a Riemann integral. We write it as
\begin{equation*}
\int_{0}^{\epsilon(K-1)}\dif y\frac{1}{\epsilon}\int_{\delta+\epsilon\lfloor y/\epsilon \rfloor}^{\delta+\epsilon\lfloor y/\epsilon \rfloor +\epsilon}\dif x\,\frac{1}{\sqrt{x}}\frac{1}{\sqrt{\delta+\epsilon\lfloor y/\epsilon \rfloor}}(1-x)_+^{\theta-1}\mathbb{P}\left(L^{(\infty)} \leq \frac{\delta+\epsilon\lfloor y/\epsilon \rfloor}{1-x}\right)
\end{equation*}
where we have dropped the Riemann integral approximation terms, and where we have incorporated the $(\cdot \wedge 1)$ boundary condition into the integrand.
In the $x$ integral make the change of variables $x' = (x-\delta-\epsilon\lfloor y/\epsilon \rfloor)/\epsilon$. Then the sum over $k$ in \eqref{ksumapprox} is equal to
\begin{equation}
\begin{split}
&\int_{0}^{\epsilon(K-1)}
\dif y
\int_{0}^{1}
\dif x'\,\frac{1}{\sqrt{\delta+\epsilon\lfloor y/\epsilon \rfloor+x'\epsilon}}\frac{1}{\sqrt{\delta+\epsilon\lfloor y/\epsilon \rfloor}}\\
&\times(1-\delta-\epsilon\lfloor y/\epsilon \rfloor-x'\epsilon)_+^{\theta-1}\mathbb{P}\left(L^{(\infty)} \leq \frac{\delta+\epsilon\lfloor y/\epsilon \rfloor}{1-\delta-\epsilon \lfloor y/\epsilon \rfloor-x'\epsilon}\right).
\end{split}
\end{equation}
Note that for every fixed $\delta>0$, the integrand above is bounded by a constant times $(1-\delta-\epsilon\lfloor y/\epsilon \rfloor-x'\epsilon)^{\theta-1}$, which is integrable for all $\theta>0$. Note that this is needed because at $\lfloor y/\epsilon \rfloor =(K-1)$ and $x'=1$ this term is singular. Furthermore, the integrand converges pointwise as $\epsilon \to 0$ to
\begin{equation}
\frac{1}{\delta+y}(1-\delta-y)^{\theta-1}\mathbb{P}\left(L^{(\infty)} \leq \frac{\delta+y}{1-\delta-y}\right).
\end{equation}
Then application of the dominated convergence theorem shows that 
\begin{equation}
\mathcal{S}_{2} \sim \frac{n^{\theta-1}}{\Gamma(\theta)}\,\theta\int_{0}^{1-\delta}\frac{1}{\delta+y}(1-\delta-y)^{\theta-1}\mathbb{P}\left(L^{(\infty)} \leq \frac{\delta+y}{1-\delta-y}\right)\dif y.
\end{equation}
The substitution $u=\delta+y$ shows that this integral is exactly the constant $C_{\delta}$ in \eqref{cdeldef}. Hence \eqref{goalsj} is verified for $j=2$. The calculation of $\mathcal{S}_{3}$ is similar and gives, as $n \to \infty$,
\begin{equation}
\mathcal{S}_{3} \sim \frac{n^{\theta-1}}{\Gamma(\theta)}\,\theta\sum_{k=0}^{K-1}\int_{\delta+k\epsilon}^{\delta+(k+1)\epsilon}\frac{1}{\delta+k\epsilon}(1-x)_+^{\theta-1}\mathbb{P}\left(L^{(\infty)} \leq \frac{\delta+k\epsilon}{1-x}\right)\,\dif x, \label{s3}
\end{equation} 
for fixed $\delta, \epsilon >0$. Then taking $\epsilon \to 0$ by a similar Riemann sum approximation, we obtain \eqref{goalsj} for $j=3$. This completes the proof of the Lemma.
\end{proof}

\subsection{Martingale convergence}
We fix a polynomial 
\[
p(z) = \sum_{r=0}^\ell d_r z^r  
\]
with complex coefficients throughout this section.
Define 
\[
  \mtp 
  \coloneqq 
  \frac{1}{2\pi}
  \int_0^{2\pi} 
  |p(e^{i\vartheta})|^2
  \mathrm{GMC}_\theta(\dif \vartheta).
\]
As a consequence of Lemma \ref{le:martingalerep2024}, it suffices to show 
\begin{equation}
  \tilde{X}_{n}^{(\delta,\epsilon)}[p]
  \coloneqq   \sum_{k=0}^{K-1}
  \sum_{q = n_{k}+1}^{n_{k+1}}
  \mathcal{N}_{q}
  \sqrt{\frac{\theta}{n_k}}\, 
  \sum_{s=0}^\ell d_s c_{n-q+s,n_{k}},
  \label{martingalesum2025}
\end{equation}
converges in law since 
\[
  \lim_{\delta\to 0}
  \lim_{\epsilon \to 0}
  \lim_{n\to\infty}
  \frac{
    \mathbb{E}(
    |\tilde{X}_{n}^{(\delta,\epsilon)}[p]
    -{X}_{n}[p]|^2
    )
  }
  {\mathbb{E}(|c_{n}|^{2})}
  =0,
\]
\begin{remark}
  Note that to properly apply Lemma \ref{le:martingalerep2024} in this case, we would rather need a version in which $n_k$ is replaced by
  \[
  n_k(s) \coloneqq  (\lfloor \delta (n + s) \rfloor  +k \lfloor \epsilon (n + s) \rfloor) \wedge (n+s),
  \]
  for $s \in \mathbb{N}_0$ is fixed.  The details of the proof are unchanged.
\end{remark}

As $\tilde{X}_n$ is a complex martingale (note, not in the index $n$ but rather viewing the sum over $q$ as a sum of martingale increments), we can define its bracket process
\begin{equation}
  \begin{split}
    \mathcal{M}_{\theta,\delta,\epsilon,n}
    &\coloneqq
    \frac{1}{\mathbb{E}(|c_{n}|^{2})}
    \,
    \sum_{k=0}^{K-1}
    \sum_{q = n_{k}+1}^{n_{k+1}}
    \frac{\theta}{n_k}\,
    \left|
    \sum_{s=0}^\ell d_s c_{n-q+s,n_{k}}
    %c_{n-q,n_{k}}
    \right|^{2}. 
    \label{condvar2024}
  \end{split}
\end{equation}

We recall from Parseval's identity 
\begin{equation}
  \begin{aligned}
    &\sum_{n=0}^{\infty}
    \left|
    \sum_{s=0}^\ell d_s 
    c_{n+s,n_k}
    \right|^{2}
    r^{2n}
    =
    \sum_{s_1,s_2}
    d_{s_1} \overline{d_{s_2}}
    \sum_{n=0}^{\infty} 
    c_{n+s_1,n_k}
    \overline{c_{n+s_2,n_k}}
    r^{2n} \\
    &=
    \sum_{s_1,s_2}
    d_{s_1} \overline{d_{s_2}}r^{-2s_2}
    \sum_{n=0}^{\infty} 
    c_{n+s_1-s_2,n_k}
    \overline{c_{n,n_k}}
    r^{2n} \\
    &=
    \frac{1}{2\pi}
    \int_{0}^{2\pi} 
    p(e^{i\vartheta}) 
    \overline{p(r^{-2}e^{i\vartheta})}
    e^{\sqrt{\theta}G_{n_k}(re^{i\vartheta})}
    \,d\vartheta,
  \end{aligned}
  \label{parseval2024}
\end{equation}
where 
\begin{equation}\label{Gn}
  G_k(z) 
  \coloneqq
  2\Re
  \left( 
  \sum_{\ell=1}^k  
  \frac{z^\ell}{\sqrt{\ell}}\mathcal{N}_\ell
  \right).
\end{equation}
We also introduce the variance of this process 
\begin{equation}
  V_k(r)
  \coloneqq 
  2
  \sum_{\ell=1}^k  
  \frac{r^{2\ell}}{{\ell}}
  =
  \Exp (  G_k(r)^2  ).
\end{equation}
Now we have the following convergence, which holds along any sequence of $r_k$ which tend to infinity:
\begin{lemma}\label{mass2024}
  Let $r_k \in [0,1)$ be any sequence with $r_k \to 1$ as $k \to \infty$.  Then 
  \[
    \left(
      \sum_{n=0}^{\infty}
      \left|
      \sum_{s=0}^\ell d_s 
      c_{n+s,k}
      \right|^{2}
      r_k^{2n}  
    \right)e^{-\theta V_k(r_k)/2}
    \Prto[k]
    \mtp.
  \]
\end{lemma}
\begin{proof}
  We introduce a two-parameter family of random variables
  \[
    X_k(r) \coloneqq 
    \left(
      \sum_{n=0}^{\infty}
      \left|
      \sum_{s=0}^\ell d_s 
      c_{n+s,k}
      \right|^{2}
      r^{2n}  
    \right)
    e^{-\theta V_k(r)/2},
    \quad k \in \N, r\in [0,1].
  \]
  We also let $X_\infty(r)$ be the limit of $X_k(r)$ as $k \to \infty$ for fixed $r\in [0,1)$, i.e.
  \[
    X_\infty(r) \coloneqq 
    \left(
      \sum_{n=0}^{\infty}
      \left|
      \sum_{s=0}^\ell d_s 
      c_{n+s}
      \right|^{2}
      r^{2n}  
    \right)
    e^{-\theta V_\infty(r)/2},
  \]
  Our goal is to show that $X_k(r_k) \Prto[k] \mtp$. 

  Using Parseval's identity, (c.f. \eqref{parseval2024} and \eqref{Gn}), $X_k(r)$ can be connected to GMC with a somewhat non-standard two-parameter regularization (note that the error term $\mathscr{P}(r)$ is killed by the diverging variance).
  However, by standard theory \cite{berestycki2017elementary,shamov2016gaussian,junnila2017uniqueness}, 
  \[
    \begin{aligned}
    &X_k(1) \Prto[k] \mtp, \\
    &X_\infty(r) \Prto[r][1]  \mtp, \\
    \end{aligned} 
  \]
  as the first is the classical Kahane-type Dirichlet-kernel regularization and the second is a Poisson-kernel regularization.
  As $\theta < 1$, these convergences additionally hold in $\operatorname{L}^1.$

  Now, we additionally have a martingale structure.  For all $r < 1$,
  \[
    \E\left( X_\infty(r) \middle \vert \filt_k \right)  = X_k(r).
  \]
  Moreover, by $\operatorname{L}^1$--convergence, on taking $r\to1$ on both sides, we can exchange limit and conditional expectation to conclude that $\{X_k(1): k\}$ is a uniformly-integrable martingale and 
  \[
    X_k(1) =   \E\left(  \mtp \middle \vert \filt_k \right).
  \]

   We have that 
  \[
    X_k(r_k) = \E\left( X_\infty(r_k) \middle \vert \filt_k \right),
  \]
  and hence 
  \[
    \begin{aligned}
    \E|X_k(r_k)-X_k(1)|
    &=\E\left| 
      \E\left( X_\infty(r_k)-\mtp \middle \vert \filt_k \right)
    \right| \\
    &\leq 
    \E\left| 
      X_\infty(r_k)-\mtp
    \right| \xrightarrow[k\to\infty]{} 0. \\
    \end{aligned}
  \]
  On the other hand, from the $\operatorname{L}^1$--convergence of $X_k(1)\to \mtp$, we conclude 
  \[
    X_k(r_k) \xrightarrow[k\to\infty]{\operatorname{L}^1} \mtp.
  \]

% The ideas in \cite[Section 2, proof of Theorem 1.9]{LOS18}, which shows $L^{1}$ convergence for seemingly more difficult regularizations, might be helpful for proving this Lemma.
\end{proof}
As a corollary, we derive: 

\begin{corollary}\label{cor2024}
  For any integer $m \geq 0$, along any sequence $n_k \to \infty$,
  \[
    \frac{1}{n_k^\theta}
    \sum_{q} e^{-q m / n_k}
    \left|
      \sum_{s=0}^\ell d_s 
      c_{q+s,n_k}
    \right|^{2}
        \Prto[n_k]
    \mtp
    \times
    \left( 
    e^{\theta \gamma_E}
    \int_0^\infty e^{-m y} p_\theta(y) \dif y
    \right)
    ,
  \]
  with $\gamma_E$ the Euler-Mascheroni constant.
\end{corollary}
\begin{proof}
  Set $r_k = e^{-m/(2n_k)}$.  Then we have 
  \begin{multline*}
    \frac{1}{n_k^\theta}
    \sum_{q} e^{-m q / n_k}
    \left|
      \sum_{s=0}^\ell d_s 
      c_{q+s,n_k}
    \right|^{2} \\
    =
    \frac{1}{n_k^\theta}
    e^{\theta V_{n_k}(r_k)/2}
    \times 
    \left( 
    \sum_{q} r_k^{2q}
    \left|
      \sum_{s=0}^\ell d_s 
      c_{q+s,n_k}
    \right|^{2}
    e^{-\theta V_{n_k}(r_k)/2}
    \right).  
  \end{multline*}
  Thus applying Lemma \ref{mass2024}, it suffices to find the limit of the constant, 
  \[
    \frac{1}{n_k^\theta}
    e^{\theta V_{n_k}(r_k)/2}
    =
    \exp
    \left(
      \theta(\gamma_E+o_{n_k}(1))
      +
      \theta
      \sum_{\ell=1}^{n_k} 
      \frac{e^{-\ell m/n_k} - 1}{\ell}
    \right),
  \]
  where $\gamma_E = \lim_{n \to \infty} \left(-\log n + \sum_{\ell=1}^n \frac{1}{\ell}\right)$ is the Euler-Mascheroni constant.

  Recalling $T_{0n}$ from \eqref{convt0n}, we have that the probability generating function of $T_{0n}$ is given by 
  \[
    \Exp\left( z^{T_{0n}} \right) = 
    \prod_{j=1}^n \Exp\left( z^{jZ_j} \right)
    = 
    \prod_{j=1}^n e^{\tfrac{\theta}{j} (z^j-1)}.
  \]
  Hence taking $z=e^{-m/n}$ and using the weak-convergence of $T_{0n}/n$, 
  \[
    \exp\left( 
      \theta
      \sum_{\ell=1}^{n} 
      \frac{e^{-\ell m/n} - 1}{\ell}
    \right)
    =
    \Exp\left( e^{-m(T_{0n}/n)} \right)
    \xrightarrow[n\to\infty]{}
    \int_0^\infty e^{-m y} p_\theta(y) \dif y,
  \]
  which completes the proof.
\end{proof}
Now using Tauberian theory, we prove the following main result:
\begin{lemma}\label{tauberian2024}
  For any $0 < \delta,\epsilon < 1$, there is a (deterministic) constant $A(\theta,\delta,\epsilon)$ so that 
  \[
    \mathcal{M}_{\theta,\delta,\epsilon,n}
    \Prto[n]
    A(\theta,\delta,\epsilon) 
    \mtp.
  \]
  Moreover, we have that 
  \[
    \lim_{\delta \to 0} 
    \lim_{\epsilon \to 0} 
    A(\theta,\delta,\epsilon) =1.
  \]
\end{lemma}
\begin{proof}
  This follows the proof of the in-probability version of Karamata's Tauberian theorem (\cite[Theorem A.2]{najnudel2023secular}), with minor changes, and so we give a sketch of the proof.  We recall \eqref{condvar2024}, in terms of which
  \[
    \begin{aligned}
    \mathcal{M}_{\theta,\delta,\epsilon,n}
    &=
    \frac{1}{\mathbb{E}(|c_{n}|^{2})}
    \,
    \sum_{k=0}^{K-1}
    \sum_{q = n_{k}+1}^{n_{k+1}}
    \frac{\theta}{n_k}\,
    \left|
      \sum_{s=0}^\ell d_s 
      c_{n-q+s,n_k}
    \right|^{2}
     \\
    &\eqqcolon 
    \frac{1}{\mathbb{E}(|c_{n}|^{2})}
    \,
    \sum_{k=0}^{K-1}
    \theta n_k^{\theta-1} X_{\epsilon,k,n}.
    \end{aligned}
  \]
  Now we let $\tilde{h}^{(k,n)}$ be the function 
  \[
    \tilde{h}^{(k,n)}(x) \coloneqq \one\{ e^{-(n-n_{k}-1)/n_{k}} \leq x < e^{-(n-n_{k+1})/n_{k}}\}.
  \]
  Then 
  \[
    X_{\epsilon,k,n}
    = 
    \frac{1}{n_k^{\theta}}
    \sum_{q = 0}^\infty 
    \tilde{h}^{(k,n)}(e^{-q/n_k})
    \left|
      \sum_{s=0}^\ell d_s 
      c_{q+s,n_k}
    \right|^{2}.
  \]

  The function $\tilde{h}^{(k,n)}(x)$ converges almost everywhere as $n\to\infty$
  \[
    \tilde{h}^{k,n}(x)
    \xrightarrow[n\to\infty]{\mathrm{a.e.}}
    \one\{ \exp\left(-\tfrac{1-\delta-k\epsilon}{\delta+k\epsilon}\right)\leq x < \exp\left(-\tfrac{1-\delta-k\epsilon-\epsilon}{\delta+k\epsilon}\right)\}.
  \]
  So we define 
  \[  
    \tilde{h}^{k,\pm}(x)
    \coloneqq
    \tilde{h}^{k,\pm,\eta}(x)
    \coloneqq 
    \one\{ \exp\left(-\tfrac{1-\delta-k\epsilon}{\delta+k\epsilon}\right) \mp \eta \leq x < \exp\left(-\tfrac{1-\delta-k\epsilon-\epsilon}{\delta+k\epsilon}\right) \pm \eta
    \},
  \]
  so that for all $\eta >0$ and all $n$ sufficiently large (depending on $\eta$) and all $0 \leq x \leq 1$
  \[
    \tilde{h}^{k,-}(x)
    \leq 
    \tilde{h}^{k,n}(x)
    \leq
    \tilde{h}^{k,+}(x).
  \]
    
  The strategy now follows standard Tauberian theory methods. For any polynomial $P$ on $[0,1]$ we have from Corollary \ref{cor2024}
  \[
    \frac{1}{n_k^\theta}
    \sum_{q} P(e^{-q  / n_k})
    \left|
      \sum_{s=0}^\ell d_s 
      c_{q+s,n_k}
    \right|^{2}
    \Prto[k]
    \mtp
    \times
    \left( 
    e^{\theta \gamma_E}
    \int_0^\infty P(e^{-y}) p_\theta(y) \dif y
    \right).
  \]

  Using a combination of piecewise linear approximation and Weierstrass approximation, we can now sandwich $\tilde{h}$ by polynomials 
  \[
    P^{-,\eta}(x) 
    \leq \tilde{h}^{k,-}(x)
    \leq \tilde{h}^{k,n}(x) 
    \leq \tilde{h}^{k,+}(x)
    \leq P^{+,\eta}(x),
  \]
  with the property that 
  \[
  \int_0^{\infty} \left(P^{+,\eta}(e^{-y}) -  P^{-,\eta}(e^{-y})\right)p_\theta(y)\dif y \xrightarrow[\eta\to 0]{} 0,
  \]
  where we have used the density of polynomials in the $\mathrm{L}^1$-space on $[0,1]$ with probability measure $p_\theta( \log(1/x))\frac{\dif x}{x}$, the density of which vanishes at $0$ by the estimate in Lemma \ref{lem:t0n}.

  We conclude taking $n\to\infty$ followed by $\eta \to 0$ that
  \[
    X_{\epsilon,k,n}
    \Prto[n]
    \mtp
    \times
    \left( 
    e^{\theta \gamma_E}
    \int_0^\infty
    \one\{\tfrac{1-\delta-k\epsilon-\epsilon}{\delta+k\epsilon}\leq y < \tfrac{1-\delta-k\epsilon}{\delta+k\epsilon}\}
    p_\theta(y) \dif y
    \right).
  \]
  Hence 
    \begin{align}
    \mathcal{M}_{\theta,\delta,\epsilon,n}
    &=
    \sum_{k=0}^{K-1}
    \frac{\theta n_k^{\theta-1}}{\mathbb{E}(|c_{n}|^{2})}
    X_{\epsilon,k,n} \\
    &\Prto[n] \Gamma(\theta+1)
    \mtp 
    e^{\theta \gamma_E}
    \sum_{k=0}^{K-1}
    (\delta + k\epsilon)^{\theta-1}
    \int_{\tfrac{1-\delta-k\epsilon-\epsilon}{\delta+k\epsilon}}^{\tfrac{1-\delta-k\epsilon}{\delta+k\epsilon}}
    p_\theta(y) \dif y. \label{convm}
    \end{align}
Now make the substitution $y=\frac{1-x}{\delta+k\epsilon}$ and apply \eqref{kingman}. This shows that the right-hand side of \eqref{convm} is $ \mathcal{M}_\theta(|p|^2)\, A(\theta,\delta,\epsilon)$ where
%\begin{align}
%&\frac{1}{\delta+k\epsilon}\int_{\delta+k\epsilon}^{\delta+k\epsilon+\epsilon}p_{\theta}\left(\frac{1-x}{\delta+k\epsilon}\right)\dif x \\
%&= e^{-\theta\gamma_{E}}\frac{1}{\Gamma(\theta)} (\delta+k\epsilon)^{-\theta}\int_{\delta+k\epsilon}^{\delta+k\epsilon+\epsilon}(1-x)^{\theta-1}\mathbb{P}\left(L^{(\infty)} \leq \frac{\delta+k\epsilon}{1-x}\right)\dif x
%\end{align}
% where we used \eqref{kingman}. 
 \begin{equation}
A(\theta,\delta,\epsilon) := \theta
    \sum_{k=0}^{K-1}\frac{1}{\delta+k\epsilon}
    \int_{\delta+k\epsilon}^{\delta+k\epsilon+\epsilon}(1-x)^{\theta-1}\mathbb{P}\left(L^{(\infty)} \leq \frac{\delta+k\epsilon}{1-x}\right)\dif x \label{a-form}
 \end{equation}
Note that this is the same constant given in \eqref{s3}. Comparing with \eqref{goalsj} with $j=3$, we get $A(\theta,\delta,\epsilon) \to C(\delta)$ as $\epsilon \to 0$ and from Lemma \ref{lem:cdef} we have $C(\delta) \to 1$ as $\delta \to 0$. We conclude that
\begin{equation}
\lim_{\delta \to 0}\lim_{\epsilon \to 0}A(\theta,\delta,\epsilon) = 1.
\end{equation}

%Now we again perform a Riemann sum approximation, as $\epsilon \to 0$.  Set $t_k = \delta + k\epsilon$.
%  Then using Lemma \ref{lem:t0n}, we have that
%  \[
%    \frac{t_k}{\epsilon}
%    \int_{1}^{1+ \tfrac{\epsilon}{\delta+k\epsilon}}
%    p_\theta(y) \dif y
%    \xrightarrow[\epsilon \to 0]{} \frac{e^{-\gamma_E \theta}}{\Gamma(\theta)},
%  \]
%  from which it follows that 
%  \[
%    A(\theta,\delta,\epsilon)
%    \xrightarrow[\epsilon \to 0]{} 
%    \theta 
%    \int_{\delta}^{1} ( t)^{\theta-2} \dif t
%    \xrightarrow[\delta \to 0]{} 
%    \frac{\theta}{\theta-1}.
%  \]

\end{proof}

\subsection{Proof of main theorem using the martingale CLT}

\begin{proof}[Proof of Theorem \ref{th:l1phasep}]
  The details of the proof now follow in the same way as \cite{najnudel2023secular}.  We apply martingale central limit theorem of \cite{HH80}.  The Lindeberg condition of the CLT follows from computation of the sum of conditional fourth moments of the martingale increments (see \cite[Lemma 3.7]{najnudel2023secular}).
  %\textcolor{red}{Say something? Because of the $p$?}  
  Hence it follows that for any $\delta, \epsilon$ 
  \[
    \frac{\tilde{X}_{n}^{(\delta,\epsilon)}[p]}{\sqrt{\mathbb{E}(|c_{n}|^{2})}}
    \Wkto[n] \sqrt{A(\theta,\delta,\epsilon)\mtp} N_{\C},
  \]
  for a complex normal random variable $N_\C$ independent of the mass $\mathcal{M}_\theta.$  
  Hence, letting $\varrho$ be a metric which metrizes weak convergence, we have
  \[
    \lim_{\delta \to 0}
    \limsup_{\epsilon \to 0}
    \limsup_{n \to\infty }
    \varrho\left(\frac{\tilde{X}_{n}^{(\delta,\epsilon)}[p]}{\sqrt{\mathbb{E}(|c_{n}|^{2})}},\sqrt{\mtp} N_{\C}\right)
    =0.
  \]
  On the other hand, from Lemma \ref{le:martingalerep2024}, we also have 
  \[
    \lim_{\delta \to 0}
    \limsup_{\epsilon \to 0}
    \limsup_{n \to\infty }
    \varrho\left(\frac{\tilde{X}_{n}^{(\delta,\epsilon)}[p]}{\sqrt{\mathbb{E}(|c_{n}|^{2})}}, \frac{X_{n}[p])}{\sqrt{\mathbb{E}(|c_{n}|^{2})}}\right)
    =0,
  \]
  from which it follows 
  \[
  \frac{X_{n}[p]}{\sqrt{\mathbb{E}(|c_{n}|^{2})}} \Wkto[n]  \sqrt{\mtp} N_{\C}.
  \]
\end{proof}
 \subsection{Joint convergence of the shifted coefficients}
 \begin{proof}[Proof of Corollary \ref{cor:shiftedcoeffs}]
 By the Cramer-Wold device, it suffices to consider the convergence of the characteristic function, for real constants $\{a_k\}$ and $\{b_k\}$.
  \[
    \varphi_n \coloneq
    \Exp 
    \left(
      \exp\left(
        \frac{i}{{\sqrt{\mathbb{E}(|c_{n}|^{2})}}} 
        \sum_{k=0}^\ell (a_k \Re c_{n+k} + b_k \Im {c}_{n+k})
      \right)
    \right).
  \]
  Now suppose $p(z) = \sum_{k=0}^\ell d_k z^k$ for $d_k$ to be determined, and we look to make 
  \[
    \begin{aligned}
    (\Re X_n[p] + \Im X_n[p]) 
    &=
    \sum_{k=0}^\ell (a_k \Re c_{n+k} + b_k \Im {c}_{n+k}) \\
    &=
    \sum_{k=0}^\ell \frac{a_k - ib_{k}}{2} c_{n+k} 
    + \frac{a_k + ib_{k}}{2} \overline{c_{n+k}}\\
    \end{aligned}
  \]
  This is possible if we take 
  \[
    (\Re d_k + \Im d_k) = a_k 
    \quad \text{and} \quad 
    (-\Im d_k + \Re d_k) = b_k
    \Longleftrightarrow 
    2d_k = (a_k+b_k) + i(a_k-b_k).
  \]
  Thus, with this $p$, applying Theorem \ref{th:l1phasep}, we conclude 
  \[
    \begin{aligned}
    \varphi_n 
    =
    \Exp \left( 
      \exp\left(
      i \frac{(\Re X_n[p] + \Im X_n[p])}{\sqrt{\mathbb{E}(|c_{n}|^{2})}}
      \right) \right)
    \to 
    &\Exp \left( 
      \exp\left(
        i 
        ( \sqrt{\mtp})
        (\Re \mathcal{Z}   + \Im \mathcal{Z}) 
      \right) \right) \\
      =& \Exp \left( 
        \exp\left(
        -\mtp/2
        \right) \right) \eqqcolon \varphi_\infty.
    \end{aligned}
  \]
  Now we have that 
  \[
    \mtp 
    =\frac{1}{2\pi} 
    \int_{0}^{2\pi}
    \sum_{k_1,k_2} d_{k_1} \overline{d_{k_2}} e^{i \vartheta (k_1-k_2)}
    \mathrm{GMC}_{\theta}(\dif \vartheta).
  \]
  Note that 
  \[
    \begin{aligned}
    d_{k_1}\overline{d_{k_2}} &=\frac{1}{4}
    \left( (a_{k_1}+b_{k_1}) + i(a_{k_1}-b_{k_1}) \right)
    \left( (a_{k_2}+b_{k_2}) - i(a_{k_2}-b_{k_2}) \right) \\
    &=
    \frac{1}{2}
    \left( a_{k_1} - i b_{k_1}\right)
    \left( a_{k_2} + i b_{k_2}\right).
    \end{aligned}
  \]
  Now let, as in the statement, $\mathcal{H}$ be the Toeplitz matrix 
  \[
    \mathcal{H}_{k_1,k_2} = \frac{1}{2\pi} 
    \int_{0}^{2\pi}
    e^{i \vartheta (k_1-k_2)}
    \mathrm{GMC}_{\theta}(\dif \vartheta),
    \quad 0 \leq k_1,k_2 \leq \ell,
  \]
  which is positive semidefinite. Then define the random vector 
  \[
    (\widehat{c}_0, 
    \widehat{c}_1,
    \dots, 
    \widehat{c}_\ell)^{\top}
    \coloneqq 
    \sqrt{\mathcal{H}}( 
      \mathcal{Z}_0, 
      \mathcal{Z}_1, \dots,
      \mathcal{Z}_\ell  
      )^{\top}.
  \]
  Now if we use the complex normal structure, and the computations derived above
  \[
    \begin{aligned}
      &\Exp 
      \exp 
      \left( 
        i \sum_{k=0}^\ell (a_k \Re \widehat{c}_{k} + b_k \Im \widehat{c}_{k})
      \right) \\
      &=
      \Exp 
      \exp 
      \left(
        i 
        \sum_{k=0}^\ell \frac{a_k - ib_{k}}{2} \widehat{c}_{k} 
        + \frac{a_k + ib_{k}}{2} \overline{\widehat{c}_{k}}
        \right)\\
      &= 
      \Exp 
      \exp 
      \left(
        \Exp 
        \left[ 
        \left( 
        i 
        \sum_{k=0}^\ell \frac{a_k - ib_{k}}{2} \widehat{c}_{k} 
        + \frac{a_k + ib_{k}}{2} \overline{\widehat{c}_{k}}
        \right)^2 
        \middle\vert 
        \mathrm{GMC}_{\theta}
        \right]
        \right)\\
        &= 
        \Exp 
        \exp 
        \left(
          -
          \sum_{k_1,k_2}^\ell 
          \frac{
            (a_{k_1} - ib_{k_1})
            (a_{k_2} + ib_{k_2})
          }{4} \mathcal{H}_{k_1,k_2}
        \right)\\
        &= 
        \Exp 
        \exp 
        \left(
          -\mtp/2
        \right).
    \end{aligned}
  \]
  This completes the proof.
\end{proof}

\bibliographystyle{alpha}
\bibliography{convergence-zero-secular}
\end{document}